\documentclass{amsart}
\usepackage{amssymb}
\usepackage{amsfonts}
\usepackage{amssymb}
\usepackage{amsmath, accents}
\usepackage{amsthm}
\usepackage{float}
\usepackage{enumerate}
\usepackage{tabularx}
\usepackage{centernot}
\usepackage{mathtools}
\usepackage{stmaryrd}
\usepackage{amsthm,amssymb}
\usepackage{etoolbox}
\makeatletter
\@ifundefined{c@unusedcounter}{}{
  \@addtoreset{unusedcounter}{section}
}
\makeatother
\usepackage{tikz}
\usepackage{amssymb}
\usetikzlibrary{matrix}
\usepackage{tikz-cd}
\usepackage{tikz}
\usepackage{url}
\usepackage{marginnote}
\definecolor{mygray}{gray}{0.85}
\usepackage[backgroundcolor=mygray,colorinlistoftodos,prependcaption,textsize=small]{todonotes}
\usepackage{xcolor}
\newcommand{\mbf}{\mathbf}

\newcommand{\op}{\operatorname}

\renewcommand{\Join}{\bigvee}

\renewcommand{\leq}{\leqslant}
\renewcommand{\geq}{\geqslant}
\renewcommand{\nleq}{\nleqslant}
\renewcommand{\ngeq}{\ngeqslant}

\theoremstyle{definition}
\newtheorem{thm}{Theorem}[section]

\newtheorem{cor}[thm]{Corollary}
\newtheorem{lm}[thm]{Lemma}
\newtheorem{notation}[thm]{Notation}

\newtheorem{prob}[thm]{Problem}

\newtheorem{fact}[thm]{Fact}
\newtheorem{definition}[thm]{Definition}
\newtheorem{remark}[thm]{Remark}

\date{\today}

\begin{document}

\author{J.\,B. Nation}

\address{Department of Mathematics, University of Hawaii, Honolulu, HI 96822.}
\email{jb@math.hawaii.edu}

\author{Gianluca Paolini}

\address{Department of Mathematics ``Giuseppe Peano'', University of Torino, Via Carlo Alberto 10, 10123, Italy.}
\email{gianluca.paolini@unito.it}

\begin{abstract}  In \cite{NaPaII} we proved that the universal theory of infinite free lattices is (algorithmically) decidable, leaving open the problem of decidability of the full theory of an (infinite) free lattice. We solve this problem by proving that, for every cardinal $\kappa \geq 3$, the first-order theory of the free lattice $\mbf F_\kappa$ is undecidable.
\end{abstract}

\thanks{Research of the second named author was  supported by project PRIN 2022 ``Models, sets and classifications", prot. 2022TECZJA, and by INdAM Project 2024 (Consolidator grant) ``Groups, Crystals and Classifications''.}

\title[Elementary properties of free lattices III]{Elementary properties of free lattices III: Undecidability of the full theory}

\maketitle  

\section{Introduction} 

    The question of algorithmic decidability of a given first-order theory is a classical theme of mathematical logic. Starting from the undecidability of $\mathrm{Th}((\mathbb{N}, +, \cdot))$, the field is riddled with (un)decidability results of first-order theories. Famous examples are the undecidability of $\mathrm{Th}((\mathbb{Z}, +, \cdot))$, of the theory of groups, and of $\mathrm{Th}((\mathbb{Q}, +, \cdot))$.   For an example on the positive side, Tarski proved in \cite{tarski} that the theory $\mathrm{Th}((\mathbb{R}, +, \cdot, <))$ is decidable. It is still an open problem (posed by Tarski) whether the theory $\mathrm{Th}((\mathbb{R}, +, \cdot, <, \mathrm{exp}))$ is decidable. In another direction, Quine proved in~\cite{quine} that the first-order theory of a non-cyclic free semigroup is undecidable.

\medskip 
    In \cite{NaPaI} we began the study of the model theory of free lattices, proving several fundamental results. This was continued in \cite{NaPaII}, where we proved that the universal (equiv.~existential) theory of infinite free lattices is decidable. One of the main open questions left in \cite{NaPaII} was the question of decidability of the (full) first-order theory of a given finitely generated free lattice $\mbf F_n$ ($n \geq 3$). In this paper we resolve this problem, proving, more strongly, the following three undecidability results:

    \begin{thm}\label{main_th} \begin{enumerate}[(1)]
    \item The first-order theory of free lattices is undecidable.
    \item The first-order theory of finitely generated free lattices is undecidable.
    \item For every cardinal $\kappa \geq 3$, the first-order theory of $\mbf F_\kappa$ is undecidable.
\end{enumerate}        
\end{thm}

    Concerning Theorem~\ref{main_th}(2), we observe that the question of decidability of the finitely generated free structures in a given variety of algebras has a long tradition, on this question see e.g. \cite{Ershov1966, Idziak1987, Lavrov1962, Malcev1962} and references therein.

    \smallskip A few words on our proof. We rely on the undecidability of  the $\forall\exists$-theory of nice finite bipartite graphs (cf. \ref{def_bi_graph_nice}) proved by Nies in \cite[Theorem~4.7]{Nies1996}. First we observe that also the $\forall\exists$-theory of nice finite bipartite posets is undecidable and then, given a  $\forall\exists$-sentence $\varphi$ in the language of posets, we construct a sentence $\varphi_*$ in the same language (which is also the language of lattices) such that $\varphi$ is true in all finite lattices if and only if $\varphi_*$ is true in $\mbf{F}_\kappa$ (where $\kappa \geq 3$ is fixed).

\medskip 

    With this paper, which is the third in a series of papers (cf.~\cite{NaPaI, NaPaII}), we solved all the major open questions on the model theory of free lattices that we identified in~\cite{NaPaI}, apart from the following fundamental question which remains open.

    \begin{prob} Are finitely generated free lattices first-order rigid, i.e., is it the case that if $\mbf L$ is finitely generated and elementary equivalent to $\mbf F_n$, then $\mbf L \cong \mbf F_n$?
    \end{prob}

\subsection{Notation}

We will consider lattices as structures in a language $L = \{\leq\}$, i.e., in the language of posets. Throughout we use boldface to denote $k$-tuples, 
so that for example $\mbf x = (x_1, \dots, x_k)$, etc. As in \cite{the_book}, we denote lattices with boldface letters.

\section{Preliminaries}

    Before moving to our proof of undecidaiblity of free lattices, we need a few definitions and result from \cite{Nies1996}.

    \begin{definition}\label{def_bi_graph}
    By a {\em bipartite graph} $C = A \dot{\cup} B$ (disjoint union) we mean a structure in the language $L = \{S_{\mathrm{up}}, S_{\mathrm{down}}, R\}$ such that $S_{\mathrm{up}}$ and $S_{\mathrm{down}}$ are unary predicates, $R$ is binary, and we have the following:
    \begin{enumerate}[(1)]
        \item $S_{\mathrm{up}}$ holds of $c \in C$ iff $c \in A$;
        \item $S_{\mathrm{down}}$ holds of $c \in C$ iff $c \in B$;
        \item $R$ is an irreflexive relation on $C$;
        \item if $C \models cRd$, then $c \in A$ and $d \in B$.
    \end{enumerate}
\end{definition}

    \begin{definition}\label{def_bi_graph_nice} We say that a bipartite graph $C = A \dot{\cup} B$ is {\em nice} when:
\begin{enumerate}[(1)]
        \item $|A| \geq 3$,
        \item $|B| \geq 3$,
        \item $\forall a \in A$, there are at least two elements of $B$ which are adjacent to $a$, and there is at least one element of $B$ which is not adjacent to~$a$;
        \item $\forall b \in B$, there are at least two elements of $A$ which are adjacent to $b$, and there is at least one element of $A$ which is not adjacent to~$b$.
    \end{enumerate}  
\end{definition}

    \begin{fact}[{\cite[Theorem~4.7]{Nies1996}}] The $\exists\forall$-theory of finite nice bipartite graphs is undecidable.
\end{fact}

    \begin{definition}\label{def_bi_poset} By a {\em bipartite poset} we mean a poset $(C, \leq)$ such that there is a set $A$ of maximal elements, and a set $B$ of non-maximal elements, and such that the elements of $B$ are minimal.  
\end{definition}

    \begin{remark} Clearly bipartite graphs (in the sense of \ref{def_bi_graph}) and bipartite posets (in the sense of \ref{def_bi_poset}) are essentially the same thing, in fact to every bipartite graph $C$ corresponds a unique bipartite poset $Q_C$, by letting the elements from the sort $S_{\mathrm{up}}$ of $C$ be the maximal elements of the corresponding poset $Q_C$. Furthermore, $C$ and $Q_C$ are bi-definable structures (recall the choice of language in \ref{def_bi_graph}). Clearly, the correspondence goes also the other way around, i.e., to every bipartite poset $Q$ corresponds a unique bipartite graph $C_Q$ by letting the maximal elements of $Q$ be the elements of the sort $S_{\mathrm{up}}$ of $C_Q$. 
    \end{remark}

    \begin{definition}\label{def_bi_poset_nice} We say that a bipartite poset $Q$ is {\em nice} if the corresponding bipartite graph $C_Q$ is nice.
\end{definition}      
        
        We immediately deduce the following:

    \begin{cor}\label{cor_undec} The $\exists\forall$-theory of finite nice bipartite posets is undecidable.
\end{cor}

\section{Canonical joinands}

In this section, we consider elements $t$ in a free lattice that are join irreducible and not below any generator.  It turns out that such elements suffice for our purposes.

A \emph{join cover} of an element $p \in \mbf L$ is a finite subset $A \subseteq \mbf L$ such that $p \leq \Join A$.
The join cover is \emph{nontrivial} if $p \nleq a$ for all $a \in A$.
An element $p$ in a lattice $\bf L$ is \emph{join prime} if it has no nontrivial join cover, i.e., $p \leq \Join A$ implies $p \leq a$ for some $a \in A$.  
For finite subsets $A$, $B \subseteq \mbf L$, we say that $A$ \emph{refines} $B$, written $A \ll B$, if for every $a \in A$ there exists $b \in B$ such that $a \leq b$.  
The join cover $p \leq \Join A$ is \emph{minimal} if whenever $p \leq \Join B$ and $B \ll A$, then $A \subseteq B$.
In a free lattice, every nontrivial join cover refines to a minimal join cover \cite[pg.~33]{the_book}.
The join cover $p \leq \Join A$ is \emph{doubly minimal} if it is minimal (in the sense of refinement above) and whenever $p \leq \Join B \leq \Join A$ nontrivially, then $\Join B = \Join A$.

The property that $u$ is a canonical joinand of $v$ is a first order property, in fact it can be expressed by the following first-order condition: %$\op{cj}(u,v)$:  
$$\exists z (z \ne v \text{ and } u + z = v)  \text{ and } \forall a,b \ ((a+b=v) \to (u \leq a \text{ OR } u \leq b)).$$   
Recall now the relation $p \,E\, q$ if $p \ne q$ and $q$ is in a doubly minimal join cover of $p$ (cf. \cite[pg.~45]{the_book}). The relation $E$ will be crucial to our purposes. Notice that if $t$ is join irreducible (as in all the cases we will be interested in), then the predicate $t \,E\, u$ is likewise first-order, in fact it can be expressed as:  $t \,E\, u$ if $\exists v$ such that
\begin{enumerate}[(i)]
\item $t \leq u + v$;
\item $t \nleq u$ and $t \nleq v$;
\item if $r$, $s < u$ then $t \nleq r+s+v$;
\item if $t \leq y + z \leq u + v$ and $t \nleq y$ and $t \nleq z$, then $y+z=u+v$.
\end{enumerate}
%If $t$ is join irreducible, let $\op{dj}(t)$ denote the elements such that $t \,E\, u$.  
%This is likewise a first order property:  $u \in \op{dj}(t)$ if there exists $x$ such that $t \leq u + x$ properly and nontrivially, and $t \leq y + z \leq u+x$ properly and nontrivially implies $y+z=u+x$.

%Whitman II: 
%Whitman embeds $F_k$ into $F_n$ by finding $w$ such that the set of canonical joinands $\op{cj}(w)$ is a join and meet independent set (Whitman \cite{PMW1942}, Theorem~6).

%canonical joinands
%The property that $u$ is a canonical joinand of $v$ is a first order property $\op{cj}(u,v)$:  
%($\exists z \text{ such that } z \ne v \text{ and } u + z = v)  \text{ and } (\forall a,b \ (a+b=v) \to (u \leq a \text{ OR } u \leq b))$.  

%free sublattices
%more or less obvious but it's in Whitman II
%Assume $v \in F_n$ with canonical joinands $u_1, \dots, u_k$.
%If the canonical joinands of $v$ are meet independent, then they generate a copy of $\mbf F_k$ in $\mbf F_n$ (notice that the canonical joinands of an element are always join independent). 
We need two lemmas from \cite{the_book}: namely 1.18 and 3.11.

\begin{lm} \label{lm:118} Let $\mbf F$ be a free lattice. Then an element of the form
$$t = \prod_{i \in [1, k]} \sum_{j \in [1, m_j]} t_{ij} \in F$$
with $k > 1$ and, for every $j \in [1, k]$, $m_j > 1$, is in canonical form iff:
    \begin{enumerate}[(1)]
        \item each $t_i$ is a generator or a join;  %%as indicated here
        \item each $t_i$ is in canonical form;
        \item $t_i \nleq t_j$ for $i \ne j$;
        \item all $t_{ij} \ngeq t$.
    \end{enumerate}
%\JB{Yes, any free lattice.  In rewriting, make it clear that as in (1) some $t_i$ could be generators, and our construction of $w_Q$ just avoids that.  The original statement of these two lemmas in the blue book includes generators as meetands.}
\end{lm}

\begin{lm} \label{lm:311} Let $\mbf F$ be a free lattice. If $t \in F$ with $t = \prod_{i \in [1, k]} \sum_{j \in [1, m]} t_{ij}$ canonically, then the doubly minimal join covers of $t$ are precisely $$\{ \{ t_{i1}, \dots, t_{im} \} : i \in [1, k]\}.$$
That is, the elements with $t \,E\, u$ are precisely the $t_{ij}$'s.
In particular, the set 
$$\{ u \in F : t \,E\, u \}$$
is finite.
\end{lm}

%THIS PARAGRAPH MOVED UP
%The property that $u$ is a canonical joinand of $v$ is a first order property $\op{cj}(u,v)$:  
%($\exists z \text{ such that } z \ne v \text{ and } u + z = v)  \text{ and } (\forall a,b \ (a+b=v) \to (u \leq a \text{ OR } u \leq b))$.   If $t$ is join irreducible, let $\op{dj}(t)$ denote the elements such that $t \,E\, u$.  
%This is likewise a first order property:  $u \in \op{dj}(t)$ if there exists $x$ such that $t \leq u + x$ properly and nontrivially, and $t \leq y + z \leq u+x$ properly and nontrivially implies $y+z=u+x$.

%Let $Q$ be a finite bipartite poset, and assume $Q$ has more than one maximal element, and that no maximal element is above every minimal element (i.e., recall Convention~\ref{the_bipartite_convention}).  
%Let $m=|Q|$.  Form the word $w_Q$ in $F_m$ by defining:
%\[  w_Q = \prod_{a \in \op{max}(Q)} \left ( \xi(a) \ + \sum_{b \in \op{min}(Q),\  b \nleq a} \xi(b) \right %) \ .        \]

    %Recall that we consider lattices as structures in the language of posets, i.e., the language $L = \{\leq \}$.

\begin{lm} \label{lm:EFBP} Let $\mbf F$ be a free lattice. There is a first-order formula $\Psi(v)$ such that, for every $w \in F$ we have that $\mbf F \models \Psi(w)$ if and only if $w$ is a proper meet, $w \nleq x$ for every generator $x$ of $\mbf F$, and $U = \{ u : w \,E\, u \}$ is a nice bipartite poset (cf. \ref{def_bi_graph_nice}-\ref{def_bi_poset_nice}).
%\newline \gianluca{But we had to add "$w$ is a proper meet" after the iff, otherwise it would not be true! If agree with current version, remove this comment.}
\end{lm}

    \begin{proof} We will use the following abbreviations in the definition of $\Psi(v)$:
\begin{enumerate}[(i)]
    \item $u \in U$ iff $w \,E\, u$;
    %iff there exists $r$ such that $w   \leq u + r$ properly and nontrivially, and $t \leq y + z \leq u+r$ properly and nontrivially implies $y+z=u+r$;
    \item $u \in \op{max}(U)$ for $u \in U$ and $t > u$ implies $t \notin U$;
    \item $t \in \min(U)$ dually.
\end{enumerate}
Now, $\Psi(w)$ is the conjunction of the following conditions:
\begin{enumerate}[(a)]
    \item $w$ is a proper meet;
    \item if $x$ is both join and meet prime then $w \nleq x$;
    \item $\neg \  ( \exists u_1, u_2, u_3 \in U : \ u_1 > u_2 > u_3 )$;
    \item there are three distinct elements in $\max(U)$;
    \item there are three distinct elements in $\min(U)$;
    \item if $u \in \max(U)$ then there exist $s_1, s_2, s_3 \in \min(U)$ such that $u \geq s_1$ and $u \geq s_2$ and $s_1 \ne s_2$ and $u \ngeq s_3$;
    \item dually for $u \in \min(U)$.
\end{enumerate}
\end{proof}

\section{Undecidability}

    \begin{fact}\label{fact_emb_xi} If $Q$ is a finite poset with $Q = \{ q_1, \dots, q_m\}$, then there is a standard embedding $\xi : Q \rightarrow \mbf F_m(x_1, ..., x_m)$ via $\xi(q_i) = \prod \{ x_j  : q_j \geq q_i \}$.
\end{fact}

An example of the embedding from \ref{fact_emb_xi} is given in Figure~\ref{fig:xi}.

%Example of xi
\begin{figure}[H]
\begin{center}
\tikzstyle{every node}=[scale=1,draw,circle,fill=white,minimum size=5pt,inner sep=0pt,label distance=1mm] 
\begin{tikzpicture}[baseline=(current bounding box.center)]
    \node (1) at (0,2) [label=above:$x_1$] {};
    \node (2) at (2,2) [label=above:$x_2$] {};
    \node (3) at (4,2) [label=above:$x_3$] {};
    \node (4) at (6,2) [label=above:$x_4$] {};
    \node (5) at (0,0) [label={[anchor=base,yshift=-3mm]below:$x_1 x_2 x_5$}] {};
    \node (6) at (2,0) [label={[anchor=base,yshift=-3mm]below:$x_2 x_3 x_4 x_6$}] {};
    \node (7) at (4,0) [label={[anchor=base,yshift=-3mm]below:$x_3 x_4 x_7$}] {};
    \node (8) at (6,0) [label={[anchor=base,yshift=-3mm]below:$x_4 x_8$}] {};
    
    \draw (1)--(5)--(2)--(6)--(3)--(7)--(4)--(8); 
    \draw (6)--(4); 
\end{tikzpicture}
\end{center}
\caption{Example of embedding of a bipartite poset $Q$ into $F_8$.}
\label{fig:xi}
\end{figure}
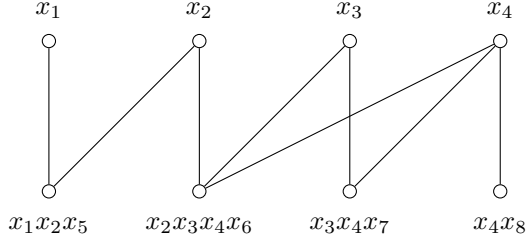

\begin{notation}\label{notation_wQ} Recalling \ref{fact_emb_xi} and $\xi$ from there. Let $Q$ be a finite nice bipartite poset. Let $m = |Q|$.  Form the word $w_Q$ in $\mbf F_m$ by defining:
\[  w_Q = \prod_{a \in \op{max}(Q)} \left ( \xi(a) \ + \sum_{b \in \op{min}(Q),\  b \nleq a} \xi(b) \right ).        \]
    \end{notation}

    For example, from Figure~\ref{fig:xi}, we have
\begin{align*}
    w_Q &= (x_1 + x_2 x_3 x_4 x_6 + x_3 x_4 x_7  + x_4x_8) \\
        &\ \   \cdot (x_2 +  x_3 x_4 x_7  + x_4x_8) \\
        &\ \ \cdot (x_3 +  x_1 x_2 x_5  + x_4x_8) \\
        &\ \ \cdot (x_4 +  x_1 x_2 x_5 ) 
\end{align*}

    \begin{lm}\label{lemma_wQ} Let $Q = A \dot{\cup} B$ be a finite nice bipartite poset. In the context of Notation~\ref{notation_wQ} and recalling the embedding $\xi$ from \ref{fact_emb_xi}. Let $\kappa \geq m = |Q|$ be a cardinal. Then the following hold:
    \begin{enumerate}[(1)]
        \item $w_Q$ as given in \ref{notation_wQ} is in canonical form (and so it is a proper meet);
        \item $\{ u \in F_\kappa : w_Q \,E\, u \} = \xi(Q)$;
        \item $\mbf F_\kappa \models \Psi(w_Q)$.
    \end{enumerate}
%\gianluca{I say directly things on $F_\kappa$. Agree? Just to be sure, please check. If OK remove this comment.}
\end{lm}

    \begin{proof} It is easy to see that it suffices to show items (1)-(3) for $\kappa = m$, i.e., for $\mbf F_\kappa = \mbf F_m$,
    since by \ref{lm:311}, $w_Q \,E\, u$ implies $\op{var}(u) \subseteq \op{var}(w_Q)$, where if $\mbf F(X)$ is a free lattice and $u \in \mbf F(X)$, with $\op{var}(u)$ we indicate the variables from the generating set $X$ occurring in the canonical form for~$u$.
    
    \smallskip \noindent Concerning item (1), apply Lemma~\ref{lm:118}.  The first three conditions for canonical forms are easy, using the fact that $|A|> 1$ and $\forall a \exists b \ a \ngeq b$.  The fourth condition holds because for every $x_k \in \{x_1, ..., x_m\}$, every meetand of $w_Q$ is above $\prod_{i \in [1, m], i \ne k} x_i$.  On the other hand, every $t_{ij}$ is $\xi(q)$ for some $q \in Q = A \dot{\cup} B$, whence $\xi(q) \leq x_k$ for at least one $k$. 
    But $\prod_{i \in [1, m], i \ne k} x_i \nleq x_k$, so $w_Q \nleq \xi(q)$.
    Thus $w_Q \nleq \xi(q)$ for all $q$.

    \smallskip \noindent Item (2) now follows immediately from (1) by Lemma~\ref{lm:311}.
 %   \gianluca{Actually maybe something more should be said here since we are not claiming something on $\mbf F_\kappa$, something like the passage from $\mbf F_m$ to $\mbf F_\kappa$ does not get new unwanted $E$ relators. You see what I mean? I think you see what I mean and adding a half a sentence might help here. Or otherwise please say something about this at the beginning of the proof.}

    \smallskip \noindent Now check item (3) using item (2).  
    Note that if $a = q_i \in \max(Q) = A$, then $\xi(a)=x_i$;
    if $b = q_j \in \min(Q) = B$, then $\xi(b)$ is the only element of $\xi(Q)$ below $x_j$.  That eliminates the possibility of a 3-element chain in $\xi(Q)$.  The remaining conditions of $\Psi(w_Q)$ reflect the assumption that $Q$ is nice.
%    \gianluca{What is written is fine but it does not explain as clearly as possible why this establishes (3), could you explain this a little better? Just to improve the exposition.}
\end{proof}

As explined in the introduction, out strategy is to use the undecidability from Corollary~\ref{cor_undec}. So, w.l.o.g. we are looking at sentences of the form
\[  \varphi: \qquad \exists \mbf x \forall \mbf y \ (S_1 \text{ OR } \dots \text{ OR } S_p  ), \]
where each $S_j = S_j(\mbf x, \mbf y)$ is a conjunction of literals $s \leq t$ or $s \nleq t$ with $s$, $t \in \{ x_1, \dots, x_k, y_1, \dots, y_\ell \}$, for $\mbf x = (x_1, ..., x_k)$ and $\mbf y = (y_1, ..., y_\ell)$. Now, given a sentence $\varphi$ in the language of posets as above, consider the following sentence $\varphi_*$ in the same language:
%\[  \forall w \ \left( \Psi(w) \to \exists \mbf x \ (\forall j : w \,E\, x_j) \ \&\  \forall \mbf y \ (\forall k : w \,E\, y_k) \to (S_1 \text{ OR } \dots \text{ OR } S_p  ) \right ). \]
\begin{multline*}
\forall w \left( \Psi(w) \to \exists \mbf x \ (\forall j : w \,E\, x_j) \right. \\
\left. \&\ \forall \mbf y \ (\forall k : w \,E\, y_k) \to (S_1 \text{ OR } \dots \text{ OR } S_p) \right).
\end{multline*}
%Of course $A \to (B \to C)$ is equivalent to $A \,\&\, B \to C$, but perhaps the meaning is easier to parse this way.

\begin{lm}\label{lm:translate} 
\begin{enumerate}[(1)]
    \item If every finite nice bipartite poset (cf.~\ref{def_bi_poset_nice}) satisfies $\varphi$, then, for every free lattice $\mbf F$, we have that $\mbf F \models \varphi_*$.
    \item If $\varphi$ fails in a finite nice bipartite poset $Q$, then, for every cardinal $\kappa \geq |Q|$ we have that $\mbf F_{\kappa}$ fails $\varphi_*$ at $w_Q$.
\end{enumerate}
\end{lm}

\begin{proof}
First assume that $\varphi$ holds in all finite nice bipartite posets.
Let $\mbf{F}$ be a free lattice and let $w \in F$. If $\Psi(w)$ fails, then $\varphi_*$ holds (since $\varphi_*$ is an implication with antecedent $\Psi(w)$). So suppose that $\Psi(w)$ holds. Then by Lemmas~\ref{lm:311} and~\ref{lm:EFBP} we have that $U = \{ u : w \,E\, u \}$ is a {\em finite} nice bipartite poset. Now, by assumption, $U$ satisfies $\varphi$ since $U$ is a finite nice bipartite poset. Hence, there exists $\mbf x \in U^k$ such that for all $\mbf y \in U^\ell$, the configuration $S_j(\mbf x, \mbf y)$ holds for some $j \in [1, p]$. That is the conclusion of $\varphi_*$, so $\mbf F \models \varphi_*$.
%\gianluca{We need to explain why $U$ is finite!!! Probably this follows that the fact that there can be only finitely many $u$ which are $E$-adjacent on the right with a given $w$ but a priori this is non-trivial, since we cannot improve $\Psi(w)$ saying "finite", so it must be true of the relation $E$ itself, otherwise we are in trouble!!!}

%\smallskip \noindent \gianluca{I went and checked in the blue book, since $E \subseteq D$ this should be true for arguments similar to the ones in the last section of our second paper. But since this point is POTENTIALLY delicate, I would rather you write the precise details needed. OK? I hope I got it right!}

%\smallskip \noindent \gianluca{I do not think that it is a problem but notice here $\mbf F$ is not assumed to be finitely generated.}

\smallskip \noindent On the other hand, suppose that $\varphi$ fails in a finite nice bipartite poset $Q$.
That is, for all $\mbf x \in Q^k$ there is a $\mbf y \in Q^\ell$ such that none of the configurations $S_j(\mbf x,\mbf y)$ occurs. Let $\kappa \geq |Q|$, then, by Lemma~\ref{lemma_wQ}, we can embed $Q$ into $\mbf F_{\kappa}$ via a map $\xi$ in such a way that we have
$$\xi(Q) = \{ u \in F_\kappa: w_Q \,E\, u \}$$
and $\mbf F_\kappa \models \Psi(w_Q)$, but the conclusion that $U = \{ u \in F_\kappa: w_Q \,E\, u \}$ contains one of the configurations $S_j$ fails. Thus, $\varphi_*$ fails at $w_Q$.
%\newline \gianluca{Slightly rewritten the proof so that it is more clear. Have a look, if OK, remove this comment.}
%looks good to me
\end{proof}

%Let $\op{Th(FBP)}$ denote the set of all sentences holding in all finite bipartite posets, and
%let $\op{Th(FN)}$ denote the set of all sentences holding in all finitely generated free lattices. 

\begin{thm} \begin{enumerate}[(1)]
    \item The first-order theory of free lattices is undecidable.
    \item The first-order theory of finitely generated free lattices is undecidable.
    %\item For every cardinal $\kappa \geq 3$, the $\forall \exists \forall$-theory of $\mbf F_\kappa$ is undecidable.
\end{enumerate}        
    \end{thm}

    \begin{proof} This is immediate from \ref{cor_undec} and \ref{lm:translate}.
    \end{proof}

%For otherwise, for each $\exists \forall$-sentence $\varphi$ in the language of posets we could decide whether $\varphi \in \op{Th(FBP)}$ by deciding whether $\varphi_*$ in $\op{Th(FN)}$.
%\newline \gianluca{Slighly rewrote the sentence above. If OK, please remove this comment.}
%JB: Agreed

\section{Whitman revisited}

We recall the Whitman embedding of $\mbf F_k$ (for $3 \leq k \leq \omega$) into $\mbf F_3$ from~\cite{PMW1942}.
Let $X_3 = \{ x_1,x_2,x_3 \}$.
To get a sublattice isomorphic to $\mbf F_4$, use $X_4 = \{ u_1, u_2, u_3, u_4 \}$ where $u_i$ are the following lattice polynomials:
\begin{align*}
    u_1 &= (x_1+x_2x_3) (x_2+ x_1  x_3)  = f_1(x_1,x_2, x_3) \\
    u_2 &= (x_1+x_2x_3) (x_3+ x_1  x_2)  = f_2(x_1,x_2, x_3) \\
    u_3 &= x_1(x_2+x_3) + x_2(x_1 + x_3) = f_3(x_1,x_2, x_3) \\
    u_4 &= x_1(x_2+x_3) + x_3(x_1 + x_2) = f_4(x_1,x_2, x_3).    
\end{align*}
To get a sublattice isomorphic to $\mbf F_5$, let $X_5 = \{ v_1, \dots, v_5 \}$ where:
\begin{align*}
    v_1 &= u_1 \\
    v_2 &= f_1(u_2,u_3,u_4) \\
    v_3 &= f_2(u_2,u_3,u_4) \\
    v_4 &= f_3(u_2,u_3,u_4) \\
    v_5 &= f_4(u_2,u_3,u_4) 
\end{align*}
Thence $X_6= \{ v_1, v_2, f_1(v_1,v_2,v_3), \dots, f_4(v_1,v_2,v_3) \}$, etc.
The independence of these elements is checked as in Whitman~\cite{PMW1942}.
(A subset $X$ of a lattice is \emph{independent} if for every $x \in X$ and every finite subset $Y \subseteq X \setminus \{ x \}$, both $x \nleq \sum Y$ and $x \ngeq \prod Y$ hold.)
Furthermore:
%\gianluca{Quickly recall what independence mean, please.}

\begin{lm} \label{lm:fkcf}
    Let $y_1$, $y_2$, $y_3$ be independent elements in a free lattice.
    Assume that each $y_iy_j$ and $y_i+y_j$ has been written in canonical form.    
    Then $f_j(y_1,y_2,y_3)$ is in canonical form for $1 \leq j \leq 4$.
\end{lm}

    \begin{remark} Given $y_1, y_2 \in \mbf F$, for $\mbf F$ a free lattice, there is no \emph{a priori} guarantee that the canonical form of say $y_1y_2$ is the meet of the canonical forms of $y_1$ and $y_2$.  
There is an algorithm to put $y_1y_2$ into canonical form, and since the element is a proper meet, its canonical form will be also, as no element in a free lattice is both a meet and a join or a generator.  That is why Lemma~\ref{lm:fkcf} is stated the way it is.
\end{remark}

\begin{proof} We only prove this for $f_1$; a similar argument applies for $f_2$, and duals for $f_3$ and $f_4$. To see that $t := f_1(y_1,y_2,y_3) = (y_1+y_2y_3) (y_2+ y_1  y_3)$ is in canonical form, apply Lemma~\ref{lm:311}.  Properties (1)--(3) of that lemma are immediate in view of the assumptions and the above remark.
    For (4), note that $y_1y_2 + y_1y_3 + y_2y_3 \leq f_1(y_1,y_2,y_3) = t$, while each $t_{ij}$ is $\leq y_1$ or $y_2$ or $y_3$.  
    By independence, $t \nleq t_{ij}$.
\end{proof}

The sets $X_4$, $X_5$, and $X_6$ were constructed above.
Continuing and taking the limit yields an infinite independent set $X_\omega = \{ z_1, z_2, z_3, \dots \}$.
Moreover, by the construction, each $z_k$ is of the form $z_k =f_1(p,q,r) = (p+qr) (q+ pr)$ for some $p, q, r \in X_{k+2}$.   Since $X_{k+2}$ is independent, $z_k$ is a proper meet.
Thus, we obtain:

%With the above construction, each $y_j$ would be a proper meet, thus join irreducible.  
%One can check that if $z_1, z_2, z_3$ are in canonical form and independent, then each $f_j(z_1,z_2,z_3)$ is in canonical form. Thus, the construction yields:
%\newline \gianluca{Sorry J.B. I do not understand this stuff about the $z_i$'s. Could you explain a little better this point and in particular: (i) who these $z_i$'s are; (ii) what is their role.}

\begin{lm} \label{lm:whitman2}
    Let $X = \{ x_1, x_2, x_3, \dots \}$, with $|X| \leq \aleph_0$.  Then there is an embedding $\zeta: \mbf {FL}(X) \rightarrow \mbf F_3$ with $z_j = \zeta(x_j)$ join irreducible with canonical form given by the Whitman construction, i.e., $z_k =f_1(p,q,r)$ for independent elements $p$, $q$, $r$.
%\gianluca{"With canonical form given by the Whitman construction" could be explain better, and maybe should.}
%\JB{ "q" and $r$ can be joins, but that is irrelevant to the conclusion.} 
\end{lm}
%\JB{Two very similar things here.
%The first is that if $z_1$, $z_2$, $z_3$ are independent and in canonical form, then each $f_j(z_1,z_2,z_3)$ is in canonical form.  Thus all $z_k$'s are in canonical form.  The second is that $\zeta\xi(w_Q)$ is in canonical form.  That uses the independence of $X_\omega$ and the nontriviality conditions of the fourth bullet, and we are assuming the stronger third bullet.}

%\JB{The first part. Assume $z_1$, $z_2$, $z_3$ are independent and in canonical form.  Put $t = f_1(z_1,z_2,z_3) = (z_1+z_2z_3) (z_2+ z_1z_3)$.    Check the conditions of Lemma~\ref{lm:118}.  Part (i) is from independence.
%For (ii), there is no \emph{a priori} reason that $z_1z_3$ and $z_2z_3$ are in canonical form, but they are proper meets, so put them in canonical form.  Part (iii) is by independence.  For (iv), 
%$t \nleq z_1$ because $z_2z_3 \leq t$, and $t \nleq z_2$ because $z_1z_3 \leq t$.  The other $f_j(\mbf z)$ follow by symmetry and duality.}        

    \begin{notation}\label{the_zeta_notation} In the previous section, we used a map $\xi$ to embed a finite nice bipartite poset $Q$ into $\mbf F_m$ where $m=|Q|$.  Then $\zeta \circ \xi$ embeds $Q$ into $\mbf F_3$, and thus into $\mbf F_3 \leq \mbf F_\kappa$ for every cardinal $\kappa \geq 3$.
\end{notation}
%Check that Lemmas \ref{lm:118} and \ref{lm:311} still apply; the arguments are similar and not hard.
%That being the case, the rest of the first section applies, using $\zeta(w_Q)$.

\begin{lm}\label{crucial_F3_lemma} In the context of Notation~\ref{the_zeta_notation}, assume $Q$ is a finite nice bipartite poset with $|Q| = m$. Then we have:
 \begin{enumerate}[(1)]
     \item $\zeta(w_Q)$ (for $w_Q(x_1, ..., x_m)$ as in Notation \ref{notation_wQ}) is in canonical form;
    \item $\{ u \in F_\kappa: \zeta(w_Q) \,E\, u \} = \{ \zeta \circ \xi(q) : q \in Q \}$;% where $Q = A \dot{\cup} B$.
     \item $\mbf F_\kappa \mbf \models \Psi(\zeta (w_Q))$ holds.
 \end{enumerate}   
\end{lm}

%%JB TO HERE 11/11

    \begin{proof} The proof is the same as that of Lemma~\ref{lemma_wQ}, but a couple of comments about canonical forms are in order.  
To prove item (1), we can apply Lemma~\ref{lm:118} to $\zeta(w_Q)$ because, for every $j \in [1, m]$, $z_j = \zeta(x_j)$ is join irreducible. 
Again, as in the proof of Lemma~\ref{lm:fkcf},
we do not know that for $b \in \min(Q)$, $\zeta \circ \xi(b) = \prod_{c_k \geq b} z_k$ is in canonical form, but it is a proper meet and hence join irreducible, thus a $t_{ij}$ in the notation of \ref{lm:118}.
As in the proof of \ref{lemma_wQ}, for every $k \in [1,m]$, every meetand of $\zeta(w_Q)$ is above $\prod_{i \in [1, m], i \ne k} z_i$.  On the other hand, every $t_{ij}$ is $\zeta \circ \xi(q)$ for some $q \in Q = A \dot{\cup} B$.
Hence $\zeta \circ \xi(q) \leq x_k$ for some $k$.  Together these facts make $\zeta(w_Q) \leq \zeta \circ \xi(q)$, that is, $t \leq t_{ij}$ is impossible. Thus, item (4) of \ref{lm:118} holds for $\zeta(w_Q)$.

\smallskip \noindent   Item (2) of this lemma follows from (1). Finally, (3) follows from (2) exactly as in the proof of \ref{lemma_wQ},
    but replacing $x_j$ by $z_j$.
    \end{proof}
    %\begin{lm}\label{crucial_F3_lemma} Let $\zeta \circ \xi$ embed $P$ into $\mbf F_\kappa$, for $\kappa \geq 3$.
    %\begin{enumerate}[(1)]
    %    \item $\zeta(w_Q)$ is in canonical form;
    %    \item $F_\kappa \models \Psi(\zeta(w_Q))$.
    %\end{enumerate}
%\gianluca{Seems like a good idea to separate this lemma, as you also suggest but I added an item.}
%\end{lm}

\begin{thm}
For every cardinal $\kappa \geq 3$, the first-order theory of $\mbf F_\kappa$ is undecidable.
\end{thm}

    \begin{proof} We argue as in the proof of \ref{lm:translate}, and in particular use the notation from there. Suppose that $\varphi$ fails in $Q$. Let $\eta = \zeta \circ \xi: Q \rightarrow \mbf F_\kappa$. 
That is, for all $\mbf x$ there is a $\mbf y$ such that none of the configurations $S_j(\mbf x,\mbf y)$ occurs. Now, by \ref{crucial_F3_lemma} we have:
$$\eta(Q) = \{ u \in F_\kappa : \zeta(w_Q) \,E\, u \},$$
and $\mbf F_\kappa \models \Psi(\zeta(w_Q))$, but the conclusion that $\{ u : \zeta(w_Q) \,E\, u \}$ contains one of the configurations $S_j$ fails. So, $\varphi_*$ fails at $\zeta(w_Q)$ in~$\mbf F_\kappa$.
\end{proof}


\begin{thebibliography}{10}

\bibitem{Ershov1966}
Ju. L. Ershov.
\newblock{\em New examples of undecidable theories} (Russian).
\newblock Algebra Logika {\bf 5} (1966), 37–47.

\bibitem{the_book}
R. Freese, J. Jezek, and J. B. Nation.
\newblock {\em Free lattices}.
\newblock American Mathematical Society, 1995.

\bibitem{Idziak1987}
P. M. Idziak.
\newblock{\em Undecidability of free pseudo-complemented semilattices}.
\newblock Publ. RIMS, Kyoto Univ. {\bf 23} (1987), 559–564.

\bibitem{Lavrov1962}
I. A. Lavrov.
\newblock{\em The undecidability of the elementary theories of certain rings} (Russian).
\newblock Algebra Logika {\bf 1} (1962), 39–45.

\bibitem{NaPaI}
J.B. Nation and G. Paolini.
\newblock{\em Elementary properties of free lattices}.
\newblock Forum Math. {\bf 37} (2025), no. 2, 581-592.

\bibitem{NaPaII}
J.B. Nation and G. Paolini.
\newblock{\em Elementary properties of free lattices II: Decidability of the universal theory}.
\newblock Preprint, Available on arXiv at: \url{https://arxiv.org/abs/2504.09128}.

\bibitem{Malcev1962}
A. I. Malcev.
\newblock{\em Axiomatizable classes of locally free algebras of certain types} (Russian).
\newblock Sib. Mat. Zh. {\bf 3} (1962), 729–743.

\bibitem{Nies1996}
A. Nies.
\newblock{\em Undecidable fragments of elementary theories}.
\newblock Algebra Universalis {\bf 35} (1996), 8-33.

\bibitem{quine}
W. V. Quine.
\newblock{\em Concatenation as a basis for arithmetic}.
\newblock J. Symb. Log. {\bf 11} (1946), 105–114.

\bibitem{tarski}
A. Tarski.
\newblock{\em A decision method for elementary algebra and geometry}. 
\newblock Univ. of California Press, 1951.

\bibitem{PMW1942}
M. Whitman.
\newblock{\em Free lattices II}.
\newblock Ann. of Math. (2) {\bf 43} (1943), 104–115.

\end{thebibliography}
\end{document}